\documentclass[12pt]{amsart}

\usepackage[utf8]{inputenc} 
\usepackage[active]{srcltx}
\usepackage{a4wide}
\usepackage{amsthm,amsfonts,amsmath,mathrsfs,amssymb}
\usepackage{dsfont}
\usepackage{mathtools}
\usepackage[T1]{fontenc}
\usepackage[utf8]{inputenc}
\usepackage{enumerate}
\usepackage{comment}
\usepackage[left=4cm,top=4cm,right=4cm,bottom=4cm]{geometry}
\numberwithin{equation}{section}

\usepackage{pgf,tikz}
\usepackage{pgfplots}

\usepackage{geometry}
\usepackage{graphicx}
\usepackage{amssymb}
\usepackage{amsmath}
\usepackage{amsthm}
\usepackage{listings}
\usepackage[colorlinks=true,urlcolor=blue,citecolor=blue,linkcolor=blue]{hyperref}
\usepackage{esint}
\usepackage{xcolor}

\usepackage{etex}
\usepackage[all]{xy}
\usepackage{pgf,tikz}
\usetikzlibrary{arrows}
\usetikzlibrary{patterns}
\usepackage{hyperref}

\newcommand{\R}{{\mathbb R}} 
\newcommand{\C}{{\mathbb C}} 
\newcommand{\N}{{\mathbb N}}

\newcommand{\D}{{\mathbb D}}

\renewcommand{\Re}{\mathrm{Re}}
\renewcommand{\Im}{\mathrm{Im}}

\newtheorem{theorem}{Theorem}[section]
\newtheorem{lemma}[theorem]{Lemma}

\newtheorem{corollary}[theorem]{Corollary}

\theoremstyle{definition}
\newtheorem{definition}[theorem]{Definition}
\newtheorem{example}[theorem]{Example}

\theoremstyle{remark}
\newtheorem{remark}[theorem]{Remark}

\numberwithin{equation}{section}

\makeatletter
\@namedef{subjclassname@2020}{%
  \textup{2020} Mathematics Subject Classification}
\makeatother

\title[Uniform convergence of continuous semigroups]{On the uniform convergence of continuous semigroups}
\date{\today}
\thanks{This research was supported in part by Ministerio de Ciencia, Innovaci\'on y Universidades (project PID2022-136320NB-I00/AEI/10.13039/501100011033/FEDER, UE)  and 
by Junta de Andaluc{\'i}a, projects P20-00664, FQM133 and FQM-104.}

\author[M.D. Contreras]{Manuel D. Contreras}
\address{Departamento de Matem\'atica Aplicada II and IMUS, Escuela T\'ecnica Superior de Ingenier\'ia, Universidad de Sevilla,
	Camino de los Descubrimientos, s/n 41092, Sevilla, Spain}
\email{contreras@us.es}

\author[C. G\'omez-Cabello ]{Carlos G\'omez-Cabello}
\address{Departamento de Matem\'atica Aplicada II and IMUS, Edificio Celestino Mutis, Avda. Reina Mercedes, s/n. 41012 - Sevilla, Universidad de Sevilla,
Sevilla, Spain}
\email{cgcabello@us.es}

\author[L. Rodr\'iguez-Piazza]{Luis Rodr\'iguez-Piazza}
\address{Departmento de An\'alisis Matem\'atico and IMUS, Facultad de Matem\'aticas, Universidad
	de Sevilla, Calle Tarfia, s/n 41012 Sevilla, Spain}
\email{piazza@us.es}

\begin{document}

\maketitle
\begin{abstract} Let $\Omega$ be a region in the complex plane $\mathbb C$ and let  $\{\Phi_t \}_{t\ge 0}$ be a 
continuous semigroup of functions on $\Omega$; that is, $\Phi_t\colon \Omega\to\Omega$
is holomorphic for every $t\ge 0$, $\Phi_0(z)=z$, for every $z\in\Omega$, $\Phi_t\circ\Phi_s=\Phi_{s+t}$, for every
$s$, $t\ge 0$, and 
\begin{equation*}\label{eso}
\Phi_t(z)\to z\,,\quad\hbox{ as $t$ goes to $0^+$,}
\end{equation*}
uniformly on compact subsets of $\Omega$. Despite this definition
only requires  the uniform convergence on compact subsets, P. Gumenyuk proved in 2014 that, when $\Omega$ is the unit disc, the convergence is uniform on the whole $\mathbb D$.
In this paper, we enhance Gumenyuk's result by proving that for every continuous semigroup $\{\Phi_t \}_{t\ge 0}$ on $\mathbb D$ we have
$$
\sup_{z\in\mathbb D} |\Phi_t(z)-z|= O(\sqrt t), \quad t\to 0^+.
$$
In addition, we provide an example showing that $O(\sqrt t)$ is the best possible rate of uniform convergence valid for all semigroups on $\mathbb D$.

When $\Omega$ is the right half-plane $\C_{+}$, we consider semigroups $\{\Phi_{t}\}$ with $\infty$ as its Denjoy-Wolff point. It is not difficult to  show that Gumenyuk's result is no longer true for these semigroups. 
Our second result characterizes when such continuous semigroups converges uniformly to the identity, as $t$ goes to zero, in terms of their infinitesimal generators.
Namely, this convergence holds if and only if the infinitesimal generator of the semigroup is bounded in the  half-plane $\{z\in \C:\, \Re z>1\}$. In this case, we can also prove that the rate of convergence is again $O(\sqrt{t})$, as $t$ goes to zero.

An example of application of this result is when the semigroup is included in the Gordon-Hedenmalm class
(the one which produces bounded composition operators on Hardy spaces of Dirichlet series), where we always have uniform convergence.

An important ingredient in the proofs of these results is the use of harmonic measures, which we have done 
through a classic result of M. Lavrentiev. 
\end{abstract}
\tableofcontents
\section{Introduction}

The study of continuous semigroups in  the unit disc $\D$ and in the right half-plane $\C_{+}$ started in the early 1900s. The work \cite{porta}, appeared in 1978, due to Berkson and Porta, meant a renewed interest in the study of continuous semigroups.  Let us recall that, given a domain $\Omega$ in the complex plane, a continuous semigroup of holomorphic self-maps of $\Omega$ --or simply a  continuous semigroup in $\Omega$-- $\{\Phi_t\}$  is a continuous homomorphism of the real semigroup $[0,+\infty)$ endowed with the Euclidean topology to the semigroup under composition of holomorphic self-maps of $\Omega$ endowed with the topology of uniform convergence on compacta.
It was proved in \cite{porta} that a continuous semigroups are  exactly the flows of a holomorphic semicomplete vector field $G$ on the domain $\Omega$; that is, the functions $t\mapsto \Phi_{t}(z)$ are the solutions of the Cauchy problems $\dot{x}(t)=G(x(t)), x(0)=z$, for all $z\in \Omega$.  

 In the last 40 years, the number of new results about this topic has grown significantly. Its connection with dynamics has been explored in a number of papers, just to mention  some recent of them we cite \cite{Bet15a,BCD,BCDGZ,BCK,pavel}.  The state of the art can be seen in \cite{manoloal}. It is also worth mentioning other in-depth connection with operators theory via semigroups of composition operators in Banach spaces of analytic functions (see, e.g., \cite{Avicou,porta,BCDMPS,ChaPar,Contreras-Diaz}).

 Analyzing the semigroups of composition operators in the disc algebra,
D\'{\i}az-Madrigal and the first author \cite{Contreras-Diaz} proved that when the functions of the semigroup have a continuous extension to the boundary of the unit disc, then the functions $\Phi_{t}$ converge to the identity, as $t$ goes to $0$, uniformly in the unit disc.
In a remarkable paper published in 2014, P. Gumenyuk \cite[Proposition 3.2]{pavel} removed the restriction of the continuous extension to the boundary and proved that given a semigroup $\{\Phi_{t}\}$ in the unit disc, the convergence to the identity, as $t$ goes to $0$, is not only uniform on compacta but also uniform in whole the unit disc. In this paper we improve such a result showing that

\begin{theorem}\label{conunifelip}
Let $\{\Phi_t\}$ be a continuous semigroup of analytic functions in the unit disc $\D$. Then, there exists $C>0$ such that
\begin{equation}
    |\Phi_t(z)-z|\leq C\sqrt{t},
    \end{equation}
    for all $z\in\D$ and for all $t>0$.
\end{theorem}

In Example \ref{Ex:cotaoptima}, it is shown that this control is sharp in the sense that the function $\sqrt t$ cannot be replaced by any function $f$ such that $f(t)=o(\sqrt{t})$, as $t$ goes to $0^{+}$. 

Gumenyuk's theorem does not hold in the right half-plane (see Example \ref{falsosem}). 
This example raises the more general question about the semigroups in $\C_{+}$ which happen to converge uniformly to identity in such half-plane and, in particular, 
this opens the way to analyze for which semigroups in the right half-plane the result holds. Our next result provides a complete characterization in terms of its infinitesimal generator when the Denjoy-Wolff point of the semigroup is $\infty$. As customary, we denote
$
\C_{\varepsilon}:=\{s\in\C:\text{Re}(s)>\varepsilon\}
$  and $\C_{+}=\C_{0}$. The space $H^{\infty}(\C_{\epsilon})$ consists on the bounded analytic functions in $\C_{\epsilon}$. Namely, we obtain:
\begin{theorem}\label{convunifsem1}
     Let $H:\C_+\to\C_+$ be a holomorphic function and denote by $\{\Phi_t\}$ the continuous semigroup whose infinitesimal generator is $H$. Then, the following statements are equivalent:
    \begin{itemize}
        \item[(i)] $H\in H^{\infty}(\C_{\varepsilon})$, for some $\varepsilon>0$.
        \item[(ii)] $H\in H^{\infty}(\C_{\varepsilon})$, for all $\varepsilon>0$.
        \item[(iii)] There exist a constant $K>0$ and $\varepsilon_{0}>0$ so that
      \begin{equation*}    \sup_{z\in\C_{\varepsilon}}|H(z)|\leq\frac{K}{\varepsilon}, \quad \text{ for all }0<\varepsilon<\varepsilon_{0}.
\end{equation*}
\item[(iv)] There exist two constants $C>0$ and $t_{0}>0$ such that if $0\leq t<t_{0}$, we have
\[
|\Phi_t(z)-z|\leq C\sqrt{t},\quad \text{for all $z\in\C_+$}.
\]
\item[(v)] $\{\Phi_t\}$  converges to the identity map uniformly on $\C_+$, as $t$ goes to zero.
\item[(vi)] There exists $\varepsilon>0$ such that $\{\Phi_t\}$ converges to the identity map
uniformly on $\C_{\varepsilon}$, as $t$ goes to zero.

\end{itemize}
\end{theorem}

To the best of the authors' knowledge, the study of continuous semigroups of analytic functions susceptible of being developable with a Dirichlet series was initiated in \cite{noi} (see also \cite{noi2}). More specifically, the continuous semigroups of analytic functions in the \emph{Gordon-Hedenmalm} class $\mathcal G$ (see Definition \ref{Def-GHclass}) were studied there. This class, introduced in \cite{gorheda}, is the family of symbols giving rise to bounded composition operators in  $\mathcal{H}^2$, the Hilbert space of Dirichlet series that Hedenmalm, Lindqvist and Seip brought up in \cite{HedLinSeip}. 
In \cite{noi}, the authors found an interesting phenomena: every continuous semigroup of analytic functions in the class $\mathcal{G}$ happens to converge to the identity uniformly in half planes  $\C_{\epsilon}$,
$\epsilon>0$. As a byproduct of above theorem we get:

\begin{corollary}\label{convunifsemg}
Let $\{\Phi_{t}\}$ be a continuous semigroup in the Gordon-Hedenmalm class,   then $\{\Phi_{t}\}$ converges to the identity map uniformly in the right half-plane.
\end{corollary}

The proof of Gumenyuk's result hinges upon  classical results of Geometric Function Theory  like  the No-Koebe-Arcs Theorem and the existence of the Koenigs map of a semigroup. In our approach,  the proofs of Theorems \ref{conunifelip} and \ref{convunifsem1}  depend strongly on good descriptions of the infinitesimal generators of continuous semigroups. In the case of the unit disc  and the right half-plane, we use the Berkson-Porta formula (see Theorem \ref{BPformula} in Section \ref{sec:semigroups}) and to prove Corollary \ref{convunifsemg} in the case of semigroups in the Gordon-Hedenmalm class  the fact that their infinitesimal generators are Dirichlet series with non-negative real part (see Theorem \ref{gcoro} in Section \ref{sectionclassg}). In the proofs, we also  strongly use  a result of  M. Lavrentiev \cite{Lav} relating  the behaviour of the length and the harmonic measure in the boundary of Jordan domains.


The plan for the paper is the following. In the next section, we recall the notion of semigroup and describe some results about its infinitesimal generator that will be needed along the paper.  Section \ref{Sec:lavrentiev} is devoted to state the result of Lavrentiev which appears to be  crucial in our proofs. Moreover, we rewrite it in terms of harmonic measures which  is more suitable in our reasoning. Section \ref{sec:half-plane}
is mainly devoted to the proof of Theorem \ref{convunifsem1}, getting in Theorem \ref{convunifsem} a stronger version of 
``(iii) implies (iv)'' of Theorem \ref{convunifsem1}.
In Section  \ref{sectionclassg} we present definitions and  the necessary results about Dirichlet series we need to prove Corollary \ref{convunifsemg}. Finally, the last section is engaged to prove  Theorem \ref{conunifelip} and we provide an example showing that such a result is sharp.

\noindent {\bf Acknowledgements.} The authors thank Athanasios Kouroupis for some significant remarks on the previous version of Theorem \ref{convunifsem1} that have greatly improved its statement.

\section{Continuous semigroups}\label{sec:semigroups}

Given a simply connected domain $\Omega$ in the complex plane $\C$, a continuous one-parameter semigroup $\{\Phi_t\}_{t\geq 0}$ of holomorphic self-maps of~$\Omega$ --a continuous semigroup of $\Omega$ for short-- is a continuous homomorphism $t\mapsto \Phi_t$ from the
additive semigroup $(\R_{\ge0}, +)$ of non-negative real numbers to the
semigroup $({\sf Hol}(\Omega,\Omega),\circ)$ of  holomorphic self-maps
of $\Omega$ with respect to composition, endowed with the
topology of uniform convergence on compacta. That is,

 \begin{definition}
      Let $\Omega$ be a domain in $\C$ and let $\{\Phi_t\}_{t\geq0}$ be a family of holomorphic functions $\Phi_t:\Omega\to \Omega$. We say that $\{\Phi_t\}_{t\geq0}$  is a continuous semigroup if:
\begin{enumerate}
\item [(i)]$\Phi_0(s)=s,$ $s\in \Omega$,
 \item[(ii)] For every $t,u\geq0$, $\Phi_t\circ\Phi_u=\Phi_{t+u}$.
    \item[(iii)] $\Phi_t\to\Phi_0$ uniformly on compact subsets of $\Omega$ as $t\to 0^+$.
\end{enumerate}
   \end{definition}
   
   For the sake of simplicity, we simply write  $\{\Phi_t\}$ to denote $\{\Phi_t\}_{t\geq0}$.

When $\Omega$ is either the unit disc or the right half-plane, 
if $\{\Phi_t\}$ is not a group of elliptic rotations, namely, it does not contain automorphisms of $\Omega$ with a fixed point in $\Omega$, then there exists a unique point $\tau\in \overline \Omega$ (the clousure in the Riemann sphere) such that $\Phi_t$ converges uniformly on compacta, as $t$ goes to $+\infty$, to the constant map $z\mapsto \tau$. Such a point $\tau$ is called the {\sl Denjoy-Wolff point} of $\{\Phi_t\}$. 
The semigroup is called {\sl elliptic} if $\tau\in  \Omega$. If $\tau\notin \Omega$, then  the semigroup is called non-elliptic.

Berkson and Porta proved in \cite{porta} the existence of the following limit
\begin{equation}\label{stress}
H(s)=\lim_{t\to0^+}\frac{\Phi_t(s)-s}{t}, \qquad \textrm{ for all } s\in \Omega
 \end{equation}
and such limit is uniform on compact sets of $\Omega$. In particular, $H$ is holomorphic.
Moreover, $t\mapsto \Phi _{t}(s)$ is the solution of the Cauchy problem:
\begin{equation}\label{cauchy}
\frac{\partial \Phi _{t}(s)}{\partial t}=H(\Phi _{t}(s))\quad
\mbox{and} \quad \Phi _{0}(s)= s\in \Omega.
\end{equation}
The function $H$ is called the  {\sl infinitesimal generator} of
the semigroup $\left\{ \Phi _{t}\right\} .$ There are several nice descriptions of the holomorphic functions that are infinitesimal generators. Maybe the most celebrated and useful one is due to Berkson and Porta, who proved the following:
\begin{theorem}\cite{porta}\label{BPformula}
Let $H:\D\to \C$  be a holomorphic function with 
$H\not\equiv 0$. Then $H$ is the 
infinitesimal generator of a continuous semigroup $\{\Phi_t\}$ if and only if there exists a unique
$\tau\in\overline \D$ and a unique $p:\D\to \C$ holomorphic with $\Re\, 
p(z)\geq 0$ such that the following formula, known as the {\sl Berkson-Porta
formula}, holds
\[
H(z)=(z-\tau)(\overline{\tau}z-1)p(z).
\]
In such a case, the point $\tau$ coincides with the Denjoy-Wolff point of the semigroup $\{\Phi_t\}$ .
\end{theorem} 

In fact, they also proved, see \cite[Theorem 2.6]{porta}, that $H$ is the infinitesimal generator of a continuous semigroup in the right half-plane with Denjoy-Wolff point $\infty$ if and only if $H(\C_{+})\subset\overline{ \C_{+}}$. 

We refer the reader to  \cite{manoloal} for all non-proven statements about continuous semigroups that we will use throughout the paper. 

\section{Preliminaries}\label{Sec:lavrentiev}

A key tool in our arguments are harmonic measures. We first give
the definition; see e.g. \cite[Chapter 4, Section 3]{Ran}.  Suppose that $\Omega$
is a domain in the complex plane with non-polar boundary and $E$ is a Borel
subset of $\partial_\infty \Omega$. The harmonic measure of $E$ relative to
$\Omega$ is the generalized Perron-Wiener solution $u$ of the
Dirichlet problem for the Laplacian in $\Omega$ with boundary values $1$ on $E$ and $0$ on $\partial_\infty \Omega\setminus E$. We will use the
standard notation
\begin{equation*}\label{hm1}
u(z)=\omega_{\Omega }(z,E),\quad z\in \Omega.
\end{equation*}
The boundary of a simply connected domain $\Omega$ contains a
continuum. Since every continuum is a non-polar set
\cite[Corollary 3.8.5]{Ran}, harmonic measures are defined for
$\Omega$. Throughout the paper the domains where harmonic measures are used will be simply connected. In this more straightforward case, a good reference for the properties we will use is \cite{Pommerenke} (see also \cite{manoloal}). Given $E\subseteq\C$, we denote by  $\ell(E)$ the outer linear measure of $E$ (see \cite[page 129]{Pommerenke}). If $E$ is a Jordan arc, then $\ell(E)$ is nothing but its length.

\begin{theorem}[Lavrentiev] \cite{Lav} (see also \cite[Proposition 6.11]{Pommerenke})\label{Lavrentiev}
Let $\Omega$ be a simply connected region in the complex plane such that $\Omega\supset\D$ and $\partial\Omega$ is Jordan curve such that $L=\ell(\partial\Omega)<\infty$. Consider $f:\D\to\Omega$ a conformal representation satisfying $f(0)=0$. Then, for every $\varepsilon>0$, there exists $\delta=\delta(\varepsilon,L)>0$ such that for every $A\subset\partial\Omega$, if $\ell(A)\leq\delta$, then $\ell(f^{-1}(A))<\varepsilon$.
\end{theorem}

In the setting of above theorem, Carath\'eodory Theorem guarantees that $f$ extends to a homeomorphism from $\overline \D$ to $\overline \Omega$ (see  \cite[Theorem 4.3.3]{manoloal} or \cite[Theorem 2.6]{Pommerenke}). 
For our purposes, it is more useful to rewrite this theorem in terms of harmonic measures. 
Notice that $\ell(f^{-1}(A))=\omega_{\D}(0,f^{-1}(A))$. Now, the conformal invariant character of the harmonic measure yields $\omega_{\D}(0,f^{-1}(A))=\omega_{\Omega}(0,A)$. That is, $\ell(f^{-1}(A))=\omega_{\Omega}(0,A)$. Taking this into account, we have the following immediate consequence.

\begin{corollary}\label{Lavrentievnoi} There exists a constant $\rho>0$ such that the following holds: Let $a\in (0,+\infty)$. Consider a simply connected domain $\Omega$ in $\C$ such that $\partial\Omega$ is a Jordan curve with $\ell(\partial\Omega)\leq 4a$ and $D(w,a/4)\subset \Omega$ for some $w\in \Omega$. Then,  $\omega_{\Omega}(w,A)<\frac18$, whenever $A$ is a measurable subset of $\partial \Omega$ with $\ell(A)\leq \rho a$. 
\end{corollary}
\begin{proof} Take $\rho=\delta(1/8,16)/4$ the constant provided by Theorem \ref{Lavrentiev}. We may assume that $w=0$. Call  $\widehat \Omega =\frac{4}{a}\Omega$. Notice that $\D\subset \widehat\Omega$ and  $\ell(\partial\widehat \Omega)\leq 16$. Let $ A$ be a measurable subset of $\partial \Omega$ with $\ell(A)\leq \rho a$. Consider $\widehat A=\frac{4}{a}A\subset \partial \widehat \Omega$. Then $\ell(\widehat A)\leq 4\rho=\delta(1/8,16)$. Therefore, 
$$
\omega_{\Omega}(0,A)=\omega_{\widehat\Omega}(0,\widehat A)<1/8.
$$
\end{proof}

\section{Continuous semigroups in the right half-plane and the proof of Theorem~\ref{convunifsem1}.}\label{sec:half-plane}


The main goal of this section is to prove Theorem~\ref{convunifsem1}. Previously, we present in Theorem \ref{convunifsem} an improvement of one of the implications in that result.

The growth rate of the modulus of the infinitesimal generator of a continuous semigroup close to the imaginary axis is intimately related to the rate of convergence of the $\text{Re}(\Phi_t(z))$ to $\text{Re}(z)$ as $t$ goes to zero and, as we are about to see in Theorem \ref{convunifsem}, it is also deeply linked to the ratio of convergence of $\{\Phi_t\}$ to the identity map. 

\ 
Before moving on, let us point out a fact which will be frequently used in this section and, in particular, in the forthcoming lemma.
    Let $H:\C_+\to\overline{\C_+}$ be the infinitesimal generator of a continuous semigroup $\{\Phi_t\}$ in $\C_+$. Then,  we have that
    \[
    \frac{d}{dt}\Re(\Phi_t(z))=\Re H(\Phi_t(z))\geq0,\quad  \text{for all $z\in\C_+$}.
    \]
    Therefore, the map $t\mapsto \Re(\Phi_t(z))$ is increasing.

\begin{lemma}\label{lemapre1}
Let $0\leq\alpha\leq1$. Let $H:\C_+\to \overline{\C_+}$ be the infinitesimal generator of a continuous semigroup $\{\Phi_{t}\}$ in $\C_+$ such that there exist a constant $K>0$ and $\varepsilon_{0}>0$ satisfying
\begin{equation*}
\sup_{z\in\C_{\varepsilon}}|H(z)|\leq\frac{K}{\varepsilon^{\alpha}}, \quad \textrm{ for all }0<\varepsilon<\varepsilon_{0}.
\end{equation*}
Then, there exists $B=B(K)>0$ and $\widetilde{t_{0}}=\widetilde{t_{0}}(\varepsilon_{0})>0$ such that if $t<\tilde{t_{0}}$, we have
\begin{equation*}
    |\emph{Re}(\Phi_t(z))-\emph{Re}(z)|\leq Bt^{
    \frac{1}{1+\alpha}}, \quad \textrm{ for all }  z\in\C_+.
\end{equation*}
\end{lemma}
\begin{proof} Fix $0<t<\varepsilon_{0}^{2}$. We can assume that $\varepsilon_0<1$. Take $\varepsilon=t^{
    \frac{1}{1+\alpha}}>0$ and choose $z\in\C_+$. If $\text{Re}(\Phi_t(z))<\varepsilon$, we are done, because $\text{Re}(\Phi_t(z))-\text{Re}(z)<\varepsilon=t^{
    \frac{1}{1+\alpha}}$. Assume then the existence of $0\leq t_1=t_1(z)<t$ such that $\text{Re}(\Phi_u(z))\geq \varepsilon$, for all $u\geq t_1$ and $\text{Re}(\Phi_u(z))\leq\varepsilon$ for every $u< t_1$. Such $t_1$ exists thanks to the fact that $u\mapsto \text{Re}(\Phi_u(z))$ is non-decreasing and continuous. Observe that $t_1$ could be $0$ and, in this case, $\Phi_{t_1}(z)-z=0$. Now, 
\begin{align*}
    \text{Re}(\Phi_t(z))-\text{Re}(z)&=\int_0^t\text{Re}(H(\Phi_u(z)))du\\
    &=\int_0^{t_1}\text{Re}(H(\Phi_u(z)))du+
    \int_{t_1}^t\text{Re}(H(\Phi_u(z)))du
    .
\end{align*}
Using the hypothesis on $H$, the definition of $t_1$ and, the choice of $\varepsilon$, we find that
\begin{align*}
    \int_0^{t_1}\text{Re}(H(\Phi_u(z)))du+
    \int_{t_1}^t\text{Re}(H(\Phi_u(z)))du
    &\leq \text{Re}(\Phi_{t_1}(z)-z)+\frac{K}{\varepsilon^{\alpha}}t\\ & \leq \varepsilon+\frac{K}{\varepsilon^{\alpha}}t=(1+K)t^{
    \frac{1}{1+\alpha}},
\end{align*}
and the conclusion follows. 
\end{proof}

\begin{lemma}\label{lemapre2}
Let $f:[c,d]\subset\R\to\R$ be a  $\mathcal C^{1}$ function of Lipschitz constant $K$. Consider $g:[c,d]\to \R$ given by $g(x)=\min\{ f(y): y\geq x\}$. Then,
\begin{itemize}
    \item[1)] The function $g$ is Lipschitz with constant $K$.
    \item[2)] If $g(x)<f(x)$, then $g'(x)=0$.
    \item[3)] If $g'(x)$ exists and $g'(x)>0$, then $g'(x)=f'(x)$.
\end{itemize}
\end{lemma}
\begin{proof}
First, notice that $g'(x)\geq0$, whenever such a value does exist, and $g(x)\leq f(x)$, for all $x$. 
For the first statement, consider $x_1,x_2\in [c,d]$ and, without loss of generality, we assume $x_1<x_2$. Take $y_{j}\geq x_{j}$ such that $f(y_{j})=g(x_{j})$, for $j=1,2$. On the one hand, if $y_{1}\geq x_{2}$, then $y_{1}=y_{2}$ and $g(x_{1})=g(x_{2})$. On the other hand,  if $y_{1}< x_{2}$, then
\begin{align*}
|g(x_{2})-g(x_{1})|=f(y_{2})-f(y_{1})&\leq f(x_{2})-f(y_{1})\leq K(x_{2}-y_{1})\leq K(x_{2}-x_{1}).
\end{align*}
So that $g$  is Lipschitz with constant $K$.

\ 
For part b), notice that if $g(x)<f(x)$, by continuity, there exists $\delta>0$ such that 
\[
\sup_{y\in(x-\delta,x+\delta)}g(y)<\inf_{y\in(x-\delta,x+\delta)}f(y).
\]
However, by the definition of $g$, this implies that $g$ is constant in such an interval and, therefore, $g'(x)=0$.

\ 
Consider $x$ a point such that $g'(x)$ exists and $g'(x)>0$. Then, by b), we have that $g(x)\geq f(x)$. However, by the definition of $g$, $g(x)\leq f(x)$. Therefore, $f(x)=g(x)$. Now, for some $\delta>0$ small enough, either $f(y)>g(y)$ for all $y\in I_{\delta}=[x,x+\delta)$, either there exists a sequence $\{y_n\}$ converging to $x$, $y_n>x$ for all $n$ and such that $f(y_n)=g(y_n)$ for all $n$. If the second case holds,  the conclusion follows since
\[
f'(x)=\lim_{y\to x^+}\frac{f(y)-f(x)}{y-x}=\lim_{n\to\infty}\frac{f(y_n)-f(x)}{y_n-x}=\lim_{n\to\infty}\frac{g(y_n)-g(x)}{y_n-x}=g'(x).
\]
Therefore, let us assume that the first case occurs. Then, by the argument to prove statement b),  we would have that $g$ is constant in $I_{\delta}$. Therefore, the right-hand derivative of $g$ at the point $x$ would be zero. Nonetheless, as $g'(x)$ exists, the left-hand derivative would also be zero, meaning that $g'(x)=0$, which contradicts the initial assumption on $x$.
\end{proof}
\begin{remark}
    According to \cite[Theorem 2.6]{porta}, the continuous semigroups $\{\Phi_t\}$ in the right half-plane with Denjoy-Wolff point $\infty$ are those whose infinitesimal generator $H$ is a holomorphic map in $\C_+$ and such that $H:\C_+\to\overline{\C_+}$. The case in which the infinitesimal generator $H$ touches the boundary of $\C_+$, this is, $H(\C_+)\cap \partial\C_+\not=\emptyset$, corresponds to the semigroup $\{\Phi_t\}$ given by $\Phi_t(z)=z+iat$, $a\in\R$, $t>0$. The statements of Theorem \ref{convunifsem} and Theorem \ref{convunifsem1} clearly hold for this case. Because of this, we shall exclude the case in which $H$ touches $\partial\C_+$ in the statments of the theorems.

\end{remark}

\begin{theorem}\label{convunifsem}
 Let $0\leq\alpha\leq1$. Let $H:\C_+\to\C_+$ be a holomorphic function and denote by $\{\Phi_t\}$ the associated continuous semigroup. Suppose that there exist a constant $K>0$ and $\varepsilon_{0}>0$ so that
      \begin{equation*}    \sup_{z\in\C_{\varepsilon}}|H(z)|\leq\frac{K}{\varepsilon^{\alpha}}, \quad \text{ for all }0<\varepsilon<\varepsilon_{0}.
\end{equation*}
     Then, there exist a constant $A=A(\varepsilon_0,\alpha,K)$ and a $t_0=t_0(\varepsilon_0,K)>0$ such that if $t<t_{0}$, we have
        \begin{equation*}
    |\Phi_t(z)-z|\leq At^{
    \frac{1}{1+\alpha}}, \quad \text{ for all } z\in\C_+.
     \end{equation*}

\end{theorem}

\begin{proof}
Step I: \emph{Simplifications}.
By Lemma \ref{lemapre1}, there exist $B=B(K)>0$  and $\widetilde{t_{0}}=\widetilde{t_{0}}(\varepsilon)>0$ such that if $t<\widetilde{t_{0}}$, then
\begin{equation}\label{estpare}
    |\Re(\Phi_t(z))-\Re(z)|\leq Bt^{
    \frac{1}{1+\alpha}}, \quad \textrm{ for all }  z\in\C_+.
\end{equation}
Therefore, it is enough to show that there exist a constant $\widetilde A$ and $t_{0}>0$ such that if $t<t_{0}$, then
\begin{equation*}
    |\Im(\Phi_t(z))-\Im (z)|\leq \widetilde A t^{
    \frac{1}{1+\alpha}}, \quad \textrm{ for all }  z\in\C_+.
\end{equation*}
Fix $z\in \C_{+}$ and $0<t<\widetilde{t_{0}}$. Let $\rho$ be the universal constant provided in Corollary \ref{Lavrentievnoi} and take $C_{1}>\max\{4, 2/\rho\}$. Write $a=|\text{Im}(\Phi_{t}(z))-\text{Im}(z)|$. Assume for the moment that $a\leq 2\varepsilon_0.$
We may suppose that
\begin{equation}\label{realmenorimaginaria}
\text{Re}(\Phi_t(z))
-\text{Re}(z) \leq\frac{1}{C_1}|\text{Im}(\Phi_t(z))-\text{Im}(z)|,
\end{equation}
otherwise taking $\tilde A>BC_1$ we are done.
 We shall carry out the proof for the case $\text{Im}(\Phi_{t}(z))-\text{Im}(z)=a$. The prove of the other case is done in a similar fashion. Define the curve
\begin{align*}
     \gamma:[0,&t]\to\C_+\\
      &u\mapsto \Phi_u(z).
        \end{align*}
We denote by $\gamma^*$ the image of the interval $[0,t]$ under $\gamma$. Similarly, we consider the square
\[
\mathcal{C}=\{s\in\C_+:\text{Re}(z)<\text{Re}(s)<\text{Re}(z)+a, \  \text{Im}(z)< \text{Im}(s)< \text{Im}(z)+a  \}
\]
and set $w$ to be the centre of the square $\mathcal{C}$. In fact, we can assume that the whole $\gamma^*$ lies in $\mathcal{C}$. \label{proofreduccion} Indeed,  define
\[
t_1=\sup\{u\in(0,t):\text{Im}(\Phi_u(z))<\text{Im}(z)\},
\]
\[
 t_2=\inf\{u\in(t_1,t):\text{Im}(\Phi_u(z))>\text{Im}(z)+a\}.
\]
Then, for all $\tau\in[t_1,t_2]$, $\text{Im}(z)\leq\text{Im}(\Phi_{\tau}(z))\leq\text{Im}(z)+a$. Let $z_0=\Phi_{t_1}(z)$ and $\upsilon=t_2-t_1$.
By the semigroup structure, $\Phi_{t_2}(z)=\Phi_{\upsilon}(z_0)$. Then,
\begin{align*}
a=\text{Im}(\Phi_{t_2}(z))-\text{Im}(\Phi_{t_1}(z))&=\text{Im}(\Phi_{t_2-t_1}(\Phi_{t_1}(z)))-\text{Im}(\Phi_{t_1}(z))\\&=\text{Im}(\Phi_{\upsilon}(z_0))-\text{Im}(z_0).
\end{align*}
On the other hand,
\[
a=\text{Im}(\Phi_t(z))-\text{Im}(z).
\]
Hence, if we prove that $\text{Im}(\Phi_{\upsilon}(z_0))-\text{Im}(z_0)\leq A\sqrt{\upsilon}$, the result will follow since
\[
\text{Im}(\Phi_t(z))-\text{Im}(z)=\text{Im}(\Phi_{\upsilon}(z_0))-\text{Im}(z_0)\leq A\sqrt{\upsilon}\leq A\sqrt{t}.
\]
Hence, we assume that $\gamma^*$ is in $\mathcal{C}$. In order to ease the reading of the proof, we define $\delta:=\text{Re}(\Phi_{t}(z))-\text{Re}(z)$. With this notation, \eqref{estpare} becomes $\delta\leq Bt^{
    \frac{1}{1+\alpha}}$. By \eqref{realmenorimaginaria} and the choice of $C_1$ we have that
    \begin{equation}\label{nimioperono}
        \delta\leq a/C_{1}<a/4.
    \end{equation}
    Also, clearly, $\text{Re}(w)>a/2$. Therefore, by hypothesis, 
\begin{equation}\label{boundepsilon2}
|H(w)|\leq \frac{2K}{a^{\alpha}}\cdot
  \end{equation}
  
\noindent Step II: \emph{Construction of $\Omega$}.  
We are going to consider the curve $\alpha:[0,t]\to\C_+$ given by
$\alpha(u)=f(u)+ig(u)$, where $f(u)=\text{Re}(\gamma(u))$ and $g(u)=\min\{\text{Im}(\gamma(s)): s\geq u\}$ which satisfies the thesis from Lemma \ref{lemapre2}, so the imaginary part of the curve $\alpha$ is non-decreasing. We shall also denote by $\alpha^*$ to $\alpha([0,t])$. The curve $\alpha^*$ divides the square $\mathcal{C}$ into two regions. Let $\Omega$ be the right hand-side region of the square (see Figure \ref{figurin}). By the definition of $\delta$, since $\delta<a/4$, we have that $\text{Re}(w)>\delta+\frac{a}{4}$ and, consequently, $D(w,a/4)\subset \Omega$.

 \begin{figure}[h!]
\includegraphics[width=\textwidth]{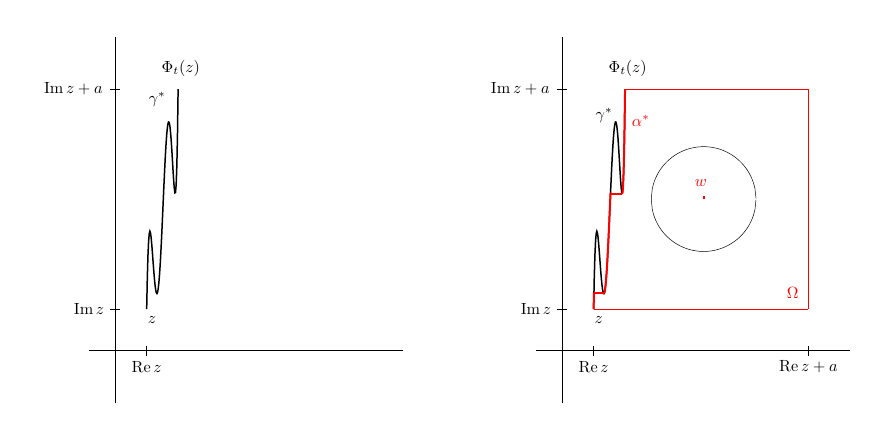}
\caption{The domain $\Omega$ and its boundary.}\label{figurin}
\end{figure}
\noindent

For every Borel set  $E\subset\partial\Omega$, we set $\omega(E)=\omega_{\Omega}(w,E)$,  where $\omega_{\Omega}(w,\cdot)$ is the harmonic measure of $\Omega$ at the point $w$. We claim that
\begin{equation}\label{estinf2}
    \omega(\alpha^*)\geq\frac14 \quad \textrm{and}  \quad\ell(\partial \Omega)\leq 4a.
\end{equation}
Indeed, for $E$ a Borel subset of $\partial\mathcal{C}$ we set $\omega_{\mathcal C}(E)=\omega_{\mathcal C}(w,E)$. Since this harmonic measure is a probability measure, $\omega_{\mathcal C}(w,\partial\mathcal{C})=1$. Furthermore, by the invariance of the harmonic measure under $\pi/2$-rotations, we have that $\omega_{\mathcal C}(w,L_{j})=1/4$, where $L_j$, stands for each of the sides of the square $\mathcal{C}$, $j=1,\ldots,4$.  Let $L_1$ be the left vertical side of the square $\mathcal{C}$. Let $\Gamma=L_{1}\cup [z+ia,z+ia+\delta]$. Notice that $\partial\mathcal{C}\setminus \Gamma=\partial\Omega\setminus\alpha^*$. Hence, by the Subordination Principle of harmonic measures, we have that
\[
\omega_{\mathcal{C}}(w,\partial\mathcal{C}\setminus \Gamma)
\geq \omega_{\Omega}(w,\partial\Omega\setminus\alpha^*).
\]
Then,
\begin{align*}
 \omega(\alpha^*) =\omega_{\Omega}(w,\alpha^*)   
  =
  1-\omega_{\Omega}(w,\Omega\setminus\alpha^*)
  &\geq 
  1-\omega_{\mathcal{C}}(w,\partial\mathcal{C}\setminus \Gamma)=  \omega_{\mathcal C}(w,\Gamma)\geq \omega_{\mathcal C}(w,L_{1})=\frac14\cdot
\end{align*}
We now show that $\ell(\partial \Omega)\leq 4a$. Since, $\Re \, \alpha $ and $\Im \, \alpha$ are non-decreasing functions, we have that
\begin{align*}\label{longitudfrontera}
   \ell(\alpha)=\int_0^{t}|\alpha'(u)|du&\leq 
    \int_{0}^{t}\text{Re}(\alpha'(u))du+\int_0^{t_n}\text{Im}(\alpha'(u))du\\ & =\text{Re}(\Phi_{t}(z)-z)+\text{Im}(\Phi_{t}(z)-z)\leq \delta+a=\ell (\Gamma).
\end{align*}
Thus, $\ell(\partial \Omega)\leq 4a$, as desired.

\noindent
Step III: \emph{Application of Lavrentiev's Theorem.}
We shall apply Corollary \ref{Lavrentievnoi} to the domain $\Omega$ and to a suitable subset $A$ of $\partial\Omega$. To this purpose, write $\eta=144$. If $\eta\geq\frac{\rho a^{1+\alpha}}{4Kt}$, then 
\begin{equation}\label{no2}
   a=\text{Im}(\Phi_{t}(z))-\text{Im}(z)\leq  \left(\frac{576K}{\rho}\right)^{\frac{1}{1+\alpha}}t^{\frac{1}{1+\alpha}}
\end{equation}
 and the conclusion would follow. Otherwise, if $\eta< \frac{\rho a^{1+\alpha}}{4Kt}$, we will reach a contradiction. Consider the sets
\[
\mathcal A_1=\{s\in\alpha^*:s\not\in\gamma^*\},\quad \mathcal A_2=\{s\in\alpha^*\cap\gamma^*:|H(s)|\leq \frac{2K}{a^{\alpha}}\eta\}.
\]
Define $\mathcal A=\mathcal A_1\cup \mathcal A_2$. By the definition of the sets $\mathcal A_1$, $\mathcal A_2$ and Lemma \ref{lemapre2} 
\begin{align*}
\ell(\mathcal A)=\ell(\mathcal A_{1})+\ell(\mathcal A_{2})&\leq \int_{\alpha^{-1}(\mathcal A_{1})}|\alpha'(u)|du+\int_{\alpha^{-1}(\mathcal A_{2})}|\alpha'(u)|du\\
&\leq
\int_{\alpha^{-1}(\mathcal A_{1})}\text{Re}(\alpha'(u))du+ \int_{\alpha^{-1}(\mathcal A_{2})}|\gamma'(u)|du \\
& \leq
\delta+ \frac{2K}{a^{\alpha}}\eta t<
\frac{a}{C_{1}}+\frac{\rho }{2} a<\rho a,
 \end{align*}
 where in the last two inequalities we have used \eqref{nimioperono} and the choices of both $\eta$ and $C_1$, respectively. Therefore, we are under the conditions of Corollary \ref{Lavrentievnoi}. Hence, the corollary gives that $\omega(\mathcal A)\leq 1/8$. This, together with \eqref{estinf2}, yields
\begin{equation}\label{Eq:otravez1/8}
    \omega(\alpha^*\setminus \mathcal A)\geq\frac18\cdot
\end{equation}
Summing up, we have shown that there exists a large set, $\alpha^*\setminus \mathcal A$, meaning that that $\omega(\alpha^*\setminus \mathcal A)\geq1/8$, where the infinitesimal generator is large, in the sense that $|H(w)|>\frac{2K}{a^{\alpha}}\eta$, for all $w$ in such set. This fact will allow us to reach the desired contradiction in the next step.

\noindent
Step IV: \emph{Reaching the contradiction.}
The function $\text{Re}(\sqrt{H})$ is positive and harmonic. Observe also that if $|\text{arg}(w)|<\pi/4$, we have that $\Re(w)\geq|w|/\sqrt{2}$. This is the reason why we consider $\sqrt{H}$ instead of $H$.  We recall that $w\in\C_{a/2}$. Hence, both the definition of harmonic measure and \eqref{boundepsilon2} respectively give
\begin{align*}
    \int_{\alpha^*\setminus \mathcal A}\text{Re}(\sqrt{H(s)})d\omega(s)\leq \int_{\partial\Omega}\text{Re}(\sqrt{H(s)})d\omega(s)
    &=
\text{Re}\left(\sqrt{H(w)}\right)\leq \sqrt2\sqrt{\frac{K}{a^{\alpha}}}\cdot
\end{align*}
On the other hand, by the previous observation, the definition of the set $\mathcal A$, estimate \eqref{Eq:otravez1/8} and using that $\eta=144$,
\begin{align*}
     \int_{\alpha^*\setminus \mathcal A}\text{Re}(\sqrt{H(s)})d\omega(s)  \geq\int_{\alpha^*\setminus\mathcal  A} \frac{\sqrt{|H(s)|}}{\sqrt{2}}d\omega(s)
    & \geq 
     \frac{1}{\sqrt{2}}\sqrt{\frac{2K\eta}{a^{\alpha}}}\omega(\alpha^*\setminus \mathcal A)\geq \frac32\sqrt{\frac{K}{a^{\alpha}}}\cdot
\end{align*}
Thus, we obtain the desired contradiction. Consequently,  $\eta\geq\frac{\rho a^{1+\alpha}}{4Kt}$, \eqref{no2} holds, and we are done in case $a\leq 2\varepsilon _0$ with $\widetilde  A:=\max\{BC_1,  \left(576K/\rho\right)^{\frac{1}{1+\alpha}}\}$ and $A(K):=\widetilde A+B$.

\
\noindent
Step V: \emph{Case $a>2\varepsilon_0$}
If $a>2\varepsilon_0$, then take $t'$ the smallest positive real number such that $$a'= \text{Im}(\Phi_{t'}(z))-\text{Im}(z)=2\varepsilon_0.$$ Then $t'<t<\tilde{t_0}$. We have already proved that $a'\leq \tilde A (t')^{\frac{1}{1+\alpha}}$. Therefore, $t\geq (2\varepsilon_0/\tilde A)^{1+\alpha}$. So we are done taking $t_0<\min\{\tilde{t_0},(2\varepsilon_0/\tilde A)^{1+\alpha}\}$.
 \end{proof}

\begin{remark}
In the above theorem, if $\alpha=0$, then $H\in H^{\infty}(\C_+)$. That is, the infinitesimal generator is bounded in the right half-plane $\C_+$. In such case, the semigroup converges to the identity like $O(t)$. In fact, for $\alpha=0$, Theorem \ref{convunifsem} is actually and equivalence. Indeed, if
  \begin{equation*}
    |\Phi_t(z)-z|\leq At, \quad \text{ for all } z\in\C_+,
  \end{equation*}
  then, this together with \eqref{stress} give
  \[
  |H(z)|=\lim_{t\to0^+}\frac{|\Phi_t(z)-z|}{t}\leq A,\quad\text{for all $z\in\C_+$.}
  \]
  As we will see in Theorem \ref{convunifsem1}, we also have an equivalence in Theorem \ref{convunifsem} for $\alpha=1$.
\end{remark}

\begin{remark}
 In connection with the previous remark, $O(t)$ is actually the best rate of convergence one could expect. Indeed, every better convergence rate to the identity of the semigroup $\{\Phi_t\}$, say $O(t^{\alpha})$, $\alpha>1$, would give the trivial semigroup. This follows immediately from \eqref{stress} 
 \[
|H(z)|\leq\lim_{t\to0^+}\frac{|\Phi_t(z)-z|}{t}\leq C\lim_{t\to0^+}t^{\alpha-1}=0,\quad z\in\C_+.
 \]
Then, $H(z)=0$ for every $z\in\C_+$ and $\{\Phi_t\}$ is the trivial semigroup. 
\end{remark}
Now we turn to the central theorem of this section. Before proving it, we establish the following auxiliary lemma. Given a simply connected domain $\Omega$, we denote by $\rho_{\Omega}$ the pseudohyperbolic distance in $\Omega$.

\begin{lemma}\label{lemmap}
    Let $z=x+iy,w=u+iv\in\C_+$ be such that $\rho_{\C_+}(z,w)<r<1$. Then,
    \[
    a) \ u\geq x\frac{1-r}{1+r},\quad b) \  |y-v|\leq \frac{2r x}{1-r^2}.
    \]
\end{lemma}
\begin{proof}
Let $\lambda>0$. The map $T_{\lambda}:\D\to\C_+$, $T_{\lambda}(z)=\lambda\frac{1+z}{1-z}$ maps $D(0,r)$, the euclidean disc of centre $0$ and radius $r$, to a euclidean disc of centre $c=\lambda(1+r^2)/(1-r^2)$ and radius $R=\lambda 2r/(1-r^2)$.
By the conformal invariance of the pseudohyperbolic distance, we have that $\rho_{\C_+}(T_{\lambda}(a),T_{\lambda}(b))=\rho_{\D}(a,b)$, for all $a,b\in\D$. 

\ 
Suppose, $z=\lambda>0$. Then, $\{\zeta\in\C_+: \rho_{\C_+}(\lambda,\zeta)<r\}$ is the image of the disc $D(0,r)$ under the map $T_{\lambda}$. Since $\rho_{\C_+}(z,w)<r$, $w$ belongs to the euclidean disc of centre $\lambda\frac{1+r^2}{1-r^2}$ and radius $\lambda \frac{2r}{1-r^2}$, so we have that
\[
u\geq \lambda \frac{1+r^2}{1-r^2}-\lambda\frac{2r}{1-r^2}=\lambda\frac{1-r}{1+r}\cdot
\]
Regarding part $b)$, since $z=\lambda $, $y=0$ and $x=\lambda$. Therefore, as the centre of the disc $T(D(0,r))$ lies on $\R$, then the imaginary part of the elements in the disc is between $-R$ and $R$. This is,
\[
|v|\leq R=\lambda\frac{2r}{1-r^2},
\]
and the lemma follows for $z$ real.

If $z\in\C$, we use the conformal invariance of $\rho_{\C_+}$ under the vertical translation $w\mapsto w-iy$.
\end{proof}
\begin{lemma}\label{lemma2}
For every $\varepsilon>0$, we have that $D(z,2\varepsilon/5)\subset D_{\rho}(z,1/4)$ for all $z\in\C_{\varepsilon}$, where $D(x,r)$ denotes the euclidean disc of centre $x$ and radius $r$ and $D_{\rho}(x,r)$ stands for the corresponding pseudo-hyperbolic disc. 
\end{lemma}
\begin{proof} 
Set $\varepsilon_0=\text{Re}(z)>\varepsilon$. Let $c=\varepsilon_0/\varepsilon>1$. Consider the disc $D_{\rho}(\varepsilon,1/4)$, which is an euclidean disc containing the point $\varepsilon$. Take
\[
\delta=
\min\{
\varepsilon-\Re(w):w\in\partial D_{\rho}(\varepsilon,1/4)\}
=
\varepsilon\left(
1-\frac{1-1/4}{1+1/4}
\right)=\frac25\varepsilon.
\]
Then, $D(\varepsilon,\delta/c)\subset D(\varepsilon,\delta)\subset D_{\rho}(\varepsilon,1/4)$. Now, taking a $c$-homothety, by the conformal invariance of the pseudohyperbolic-metric, we have that $D(\varepsilon_0,2\varepsilon/5)\subset D_{\rho}(\varepsilon_0,1/4) $. Considering the necessary translation, we obtain that $D(z,2\varepsilon/5)\subset D_{\rho}(z,1/4)$, and the lemma follows.
\end{proof}
%
\begin{proof}[Proof of Theorem \ref{convunifsem1}]
The implications (iii) implies (ii), (ii) implies (i), (iv) implies (v), and (v) implies (vi) are all trivial. Now, (iii) implies (iv) follows from Theorem \ref{convunifsem} with $\alpha=1$. Therefore, it suffices to prove (vi) implies (i) and (i) implies (iii). We begin by showing the latter one.

\ 
By hypothesis, the holomorphic function $H$ is bounded in some half-plane. Hence, there exists $\varepsilon_0>0$ so that $|H(z)|\leq M$ for every $z\in\overline{\C_{\varepsilon_0}}$. Fix $0<\varepsilon<\varepsilon_0$. Let $z\in\C_{\varepsilon}$ be such that $\Re(z)<\varepsilon_0$. Consider the point $z_0=\varepsilon_0+i\Im(z)$. Since $H:\C_+\to\C_+$ is holomorphic, by Schwarz-Pick Lemma, we have that
 \begin{align*}
    \left|
     \frac{H(z)-H(z_0)}{\overline{H(z)}+H(z_0)}
    \right|=\rho_{\C_+}(H(z),H(z_0))
    \leq\rho_{\C_+}(z,z_0)=
\left|\frac{z-z_0}{\overline{z}+z_0}\right|
=
\frac{\varepsilon_0-\Re(z)}{\varepsilon_0+\Re(z)} ,
\end{align*}
Set
\[
A=\frac{\varepsilon_0-\Re(z)}{\varepsilon_0+\Re(z)}.
\]
Then, applying the triangular inequality in the latter inequality we obtain
\begin{align*}
   |H(z)|-|H(z_0)|\leq  |H(z)-H(z_0)|\leq A|\overline{H(z)}+H(z_0)|&\leq A(|H(z)|+|H(z_0)|).
\end{align*}
This is,
\[
|H(z)|\leq\frac{1+A}{1-A}|H(z_0)|=\frac{\varepsilon_0}{\Re(z)}|H(z_0)|\leq \frac{\varepsilon_0}{\varepsilon}M.
\]
Finally, if $z\in\overline{\C_{\varepsilon_0}}$, we have that
\[
|H(z)|\leq M\leq\frac{\varepsilon_0}{\varepsilon}M
\]
and the claim follows.

\ 
We conclude showing (vi) implies (i). Let us assume that $H$ is not bounded in $\C_{\varepsilon}$. Then, given $n\in\N$, there exists $z_n\in\C_{\varepsilon}$ such that $|H(z_n)|>2n$. 
 Now, by the choice of $z_n$, one of the following cases happens:
\begin{itemize}
    \item[i)] $\text{Re}(H(z_n))\geq n$,
    \item[ii)]   $\text{Re}(H(z_n))<n$ and $\text{Im}(H(z_n))>n$, or
    \item[iii)] $\text{Re}(H(z_n))<n$ and $\text{Im}(H(z_n))<-n$.
\end{itemize}
To simplify, write $\delta=2\varepsilon/5$. Since holomorphic functions are Lipschitz continuous in the pseudo-hyperbolic metric, by Lemma \ref{lemma2}, we have that
\[
\rho(H(z_n),H(w))\leq \rho(z_n,w)<\frac14,\quad \text{for all $w\in D(z_n,\delta)$}.
\]
 \noindent \textbf{Claim.}  Either one of the following statements holds
 \begin{itemize}
     \item[a)]   $\text{Re}(H(w))>n/3$, for all $w\in D(z_n,\delta)$,
     \item[b)] $\text{Im}(H(w))>n/3$, for all $w\in D(z_n,\delta)$, or
     \item[c)] $\text{Im}(H(w))<-n/3$, for all $w\in D(z_n,\delta)$.
 \end{itemize}
  Let $t_n=\inf\{t:\Phi_t(z_n)\not\in 
 D(z_n,\delta)\}$. Using the claim, we have that  for each $0\leq t\leq t_{n}$ either $\text{Re}(H(\Phi_t(z_n)))\geq n/3$ for all $t<t_{n}$, $\text{Im}(H(\Phi_t(z_n)))\geq n/3$, for all $t<t_{n}$, or $-\text{Im}(H(\Phi_t(z_n)))\leq- n/3$, , for all $t<t_{n}$. We begin by assuming that the first holds:
   \begin{align*}
    \delta=|\Phi_{t_n}(z_n)-z_n|\geq  \Re (\Phi_{t_n}(z_n)-z_n)=\int_{0}^{t_{n}} \Re (H(\Phi_{\tau}(z_n)))\, d\tau \geq \frac13 n t_n.
\end{align*}
In the other two remaining cases the situation is similar. Hence, we can conclude that $\delta=|\Phi_{t_n}(z_n)-z_n|\geq \frac13 n t_n.$

Therefore, for each natural number $n$, there is $z_{n}\in \C_{\varepsilon}$ and $0<t_{n}\leq 3\delta/n$, such that $\delta=|\Phi_{t_n}(z_{n})-z_{n}|$. This implies that $\{\Phi_{t}\}$ cannot tend uniformly to the identity in $\C_{\varepsilon}$ as $t$ goes to zero.

  \noindent \textbf{Proof of the Claim.}  We begin by assuming case $i)$. By Lemma \ref{lemmap} a) with $r=1/4$, we have that
\[
\text{Re}(H(w))\geq \frac35\text{Re}(H(z_n))\geq \frac35 n.
\]
This gives a). Similarly, if we are under case ii) (the proof also holds for case iii)), using Lemma \ref{lemmap} b) with $r=1/4$, we find that
\[
|\text{Im}(H(w))-\text{Im}(H(z_n))|\leq \frac{2r\text{Re}(H(z_n))}{1-r^2}\leq \frac{2r}{1-r^2}n=\frac{8n}{15}.
\]
Since $\text{Im}(H(z_n))>n$, we conclude that $\text{Im}(H(w))>7n/15$ and this gives b). Arguing in a similar fashion for c) gives the desired conclusion and concludes the proof. 
\end{proof}

\

There are many interesting continuous semigroups in the right half-plane satisfying condition (ii) from Theorem \ref{convunifsem1}. More specifically, as we will see in Section \ref{sectionclassg}, there are certain continuous semigroups in $\C_+$, intimately linked to the semigroups of bounded composition operators on some Banach spaces of Dirichlet series, for which the statements of Theorem \ref{convunifsem1} hold.


\begin{example}\label{falsosem}
There exist continuous semigroups in $\C_+$ failing to converge uniformly to the identity in the whole right half-plane $\C_+$. Fix $\alpha\in[0,1)$. Consider the family of holomorphic functions in $\C_+$ given by
\[
H_{\alpha}(z)=\frac{z^{\alpha}}{1-\alpha}, \quad z\in\C_+,
\]
for defining $z^{\alpha}$ we take the principal logarithm. Clearly, $H_{\alpha}(\C_+)\subset\C_+$. Hence, $H_{\alpha}$ is the infinitesimal generator of a continuous semigroup $\{\Phi_{t}\}$ in $\C_+$. We can compute these semigroups by solving the Initial Value Problem
\begin{equation}
\dot{w}(t)=\frac{w^{\alpha}(t)}{1-\alpha}\quad
\mbox{and} \quad w(0)= z\in\C_+,
\end{equation}
where $w(t)=\Phi_t(z)$, for $z\in\C_+$ fixed. Therefore,
\[
\Phi_t(z)=(t+z^{1-\alpha})^{\frac{1}{1-\alpha}}.
\]
We can also compute the Koenigs map $h_{\alpha}$ of the semigroups $\{\Phi_{t}\}$ using that 
\[
h'_{\alpha}(z)=\frac{1}{H_{\alpha}(z)}=z^{-\alpha}(1-\alpha),\quad z\in\C_+.
\]
Therefore, $h_{\alpha}(z)=z^{1-\alpha}$ and it is straightforward to check that
\[
\Phi_{t}(z)=h^{-1}_{\alpha}(h_{\alpha}(z)+t)).
\]
Eventually, $\Phi_t(z)\to z$ as $t$ goes to zero locally uniformly. However, the infinitesimal generator $H_{\alpha}$ is not bounded in any half-plane $\C_{\varepsilon}$. Therefore, thanks to Theorem \ref{convunifsem1} the semigroup $\{\Phi_{t}\}$ fails to converge to the identity uniformly in $\C_{\varepsilon}$, $\varepsilon>0$.
\end{example}

\section{Continuous semigroups in the Gordon-Hedenmalm class}\label{sectionclassg}
For unexplained results and terminology of Dirichlet series, we refer the readers to the monographs \cite{peris,queffelecs}.
As customary, we denote by $\mathcal D$
 the space of Dirichlet series that converges somewhere, namely the series
$$
\varphi(s)=\sum_{n=1}^{\infty}a_nn^{-s},
$$
which are convergent  in some half-plane $\C_{\theta}$. 

Gordon and Hedenmalm gave the characterisation of the boundedness of composition  operators in the context of the Hardy space of Dirichlet series $\mathcal{H}^2$, see \cite{gorheda}. To do so, they introduced the nowadays known as \emph{Gordon-Hedenmalm class} $\mathcal{G}$. 

\begin{definition} \label{Def-GHclass} Let $\Phi:\C_{+}\to\C_{+}$ be an analytic function.
\begin{enumerate}
\item We say that $\Phi$ belongs to the class $\mathcal{G}_{\infty}$ if there exist $c_\Phi\in\N\cup\{0\}$ and $\varphi\in\mathcal{D}$ such that
    \begin{equation}\label{chosenone}
        \Phi(z)=c_{\Phi}z+\varphi(z).
    \end{equation}
The value $c_{\Phi}$ is known as the {\sl characteristic} of the function $\Phi$. 
\item We say that $\Phi$ belongs to the {\sl Gordon-Hedenmalm class} $\mathcal{G}$ if $\Phi\in  \mathcal{G}_{\infty}$ and  $\Phi(\C_+)\subset \C_{1/2}$ in case $c_{\Phi}=0$.
\end{enumerate}
\end{definition}

Gordon and Hedenmalm proved the following characterisation.

\begin{theorem}[Gordon-Hedenmalm]\label{gordon}
An analytic function $\Phi:\C_{1/2}\to\C_{1/2}$ defines a bounded composition operator $\mathcal{C}_{\Phi}:\mathcal{H}^2\to\mathcal{H}^2$ if and only if $\Phi$ has a holomorphic extension to $\C_{+}$ that  belongs to the class $\mathcal{G}$. 
\end{theorem}

 It is worth recalling that in \cite[Proposition 3.2]{noi}, it was shown that for any continuous semigroup $\{\Phi_t\}$ of analytic functions in the class $\mathcal{G}_{\infty}$, $c_{\Phi_{t}}=1$ for all $t$ and that its Denjoy-Wolff point is $\infty$.
 In \cite[Theorem 1.2]{noi}, it was characterized when a semigroup of composition operators is strongly continuous in the Hilbert space $\mathcal{H}^2$ as follows:
 
 \begin{theorem}\label{semigriuposGH}
Let $\{\Phi_t\}$ be a semigroup of analytic functions, such that $\Phi_t\in\mathcal{G}$ for every $t>0$ and denote by $T_{t}$ the composition operator $T_t(f)=f\circ\Phi_t$. Then, the following assertions are equivalent:
\begin{enumerate}[a)]
    \item $\{T_t\}_{t\geq0}$ is a strongly continuous semigroup in $\mathcal{H}^2$.
    \item $\{\Phi_t\}_{t\geq0}$ is a continuous semigroup.
    \item $\Phi_t(z)\to z$, as $t$ goes to $0$, uniformly in $z\in \C_{\varepsilon}$, for every $\varepsilon>0$.
\end{enumerate}
\end{theorem}

Bringing together the above Theorem \ref{semigriuposGH} and Theorem \ref{convunifsem1} we achieve Corollary \ref{convunifsemg}. It is worth also recalling that  the infinitesimal generator of a continuous semigroup in $\mathcal G$ is a holomorphic function that sends the right half-plane into its closure.   In fact, in \cite[Theorem 5.1]{noi}, it is given the following description of infinitesimal generators of the continuous semigroups in $\mathcal{G}$.

 
\begin{theorem}\label{gcoro}
Let $H:\C_+\to\overline{\C}_+$ be analytic. Then, the following statements are equivalent:
\begin{enumerate}[a)]
    \item $H$ is the infinitesimal generator of a continuous semigroup of elements in the class $\mathcal{G}$.
    \item $H\in \mathcal{D}\cap H^{\infty}(\C_{\varepsilon}),$ for all $\varepsilon>0$. 
    \item $H\in\mathcal{D}$.
\end{enumerate}
\end{theorem}

\section{Continuous semigroups in the unit disc: proof of Theorem \ref{conunifelip}.}\label{sec:unitisc}

This section is mainly devoted to the proof of  Theorem \ref{conunifelip}.  We will need the following auxiliary elementary lemma. The first statement is well-known and the reader may find a proof of it in any Complex Analysis manual.

\begin{lemma}\label{lemapre3}
Let $p:\D\to\overline{\C_+}$ be an analytic function, $G(z)=(z-1)^2p(z)$, $z\in \D$, and $G_1:\C_+\to\overline{\C_+}$ given by $G_1(w)=2p((w-1)/(w+1))$, $w\in \C_{+}$. Then,  there exists a constant $M>0$ such that the following statements hold:
\begin{enumerate}
    \item[a)] For any $z\in \D$,
    \begin{equation*}
        |p(z)|\leq\frac{M}{1-|z|^2}.
    \end{equation*}
       \item[b)] Let $\lambda\in(0,1)$ and consider $z\in H_{\lambda}=D(\lambda,1-\lambda)$. Then
    \begin{equation*}
        |G(z)|\leq\frac{M}{\lambda}.
         \end{equation*}
    \item[c)] If $w=x+iy\in \C_{+}$, then
    \begin{equation*}
        |G_1(w)|\leq M\frac{(1+x)^2+y^2}{2x}.
    \end{equation*}
\end{enumerate}
\end{lemma}

\begin{proof}
Part $a)$ is a well-known result which follows from  Herglotz's representation formula.

Regarding statement $b)$, take $z\in D(\lambda,1-\lambda)$. Then there is $\beta>\lambda$ such that $|z-\beta|=1-\beta$. That is, $z=(1-\beta)e^{i\theta}+\beta$ for some $\theta\in (0,2\pi)$. A quick computation shows $\frac{|z-1|^{2}}{1-|z|^{2}}=\frac{1-\beta}{\beta}$. Therefore, by a), 
$$
|G(z)|=|z-1|^2|p(z)|\leq |z-1|^2\frac{M}{1-|z|^2}=M\frac{1-\beta}{\beta}\leq M\frac{1}{\beta}\leq M\frac{1}{\lambda}.
$$

Claim $c)$ also follows from $a)$. Indeed, 
\begin{align*}
    |G_1(w)|=2|p((w-1)/(w+1))|\leq 2M  \frac{|w+1|^{2}}{|w+1|^{2}-|w-1|^{2}}=2M\frac{(1+x)^2+y^2}{4x}.
\end{align*}
\end{proof}

\begin{proof}[Proof of the elliptic case in Theorem \ref{conunifelip}] Let $\{\Phi_t\}$ be an elliptic continuous semigroup on $\D$. Since $|\Phi_{t}(z)-z|\leq 2$, it is enough to get the thesis if $t$ is close to zero. By a standard argument,  we may assume that its Denjoy-Wolff point is zero. 
By the Schwarz Lemma and the algebraic structure of the semigroup, we have that for every $t>0$, $|\Phi_t(z)|\leq |z|$ and $|\Phi_t(z)|$ increases to $|z|$, as $t$ decreases to $0$, for all $z\in \D$.

We shall denote by $G$ the infinitesimal generator of the semigroup $\{\Phi_t\}$. It is known that $G(z)=-zp(z)$ (see Theorem \ref{BPformula}), where $p:\D\to\overline {\C_+}$ is analytic and there exists $M>0$ such that
\begin{equation*}
    |p(z)|\leq\frac{M}{1-|z|},\quad z\in\D.
\end{equation*}
Set $C_1:=\sup_{u\in D(0,1/2)}|G(u)|$. Take $z_0\in\D$ and define $t_1=\inf\{t>0:|\Phi_t(z_0)|\leq 1/2\}$. Notice that $t_1$ is $0$ if $|z_0|\leq 1/2$. Then, by the definition of $t_1$,
\begin{align*}
    |\Phi_t(z_0)-z_0|=\left|\int_0^t
    G(\Phi_{\tau}(z_0))d\tau
    \right|&\leq \int_0^{t_1}|G(\Phi_{\tau}(z_0))|d\tau+\int_{t_1}^t|G(\Phi_{\tau}(z_0))|d\tau\\
    &\leq \int_0^{t_1}|G(\Phi_{\tau}(z_0))|d\tau+tC_1.
\end{align*}
Therefore, we can suppose that $|z_0|>1/2$ and it suffices to bound the first term in the latter inequality. 

From now on, fix $z_0=re^{i\theta_0}$ with $r\in (1/2,1)$ and $\theta_0\in [0,2\pi]$. Consider the set
$$\mathcal{A}=\{z\in\D: \frac12 <|z|<1, -\pi/2+\theta_0<\text{arg}(z)<\pi/2+\theta_0\}.$$ 
Notice that $\inf\{|z-z_0|:\, |z|>1/2, z\notin \mathcal{A}\}\geq \sqrt{2}/2$. 

We claim that there exists $C=C(M)>0$ such that if $\Phi_u (z_0)\in \mathcal A$ for all $0\leq u< t$ then $|\Phi_{t}(z_0)-z_0|\leq C\sqrt{t}$.

Let us see that this is enough to conclude the proof and then we will prove the claim.
Take $t_2\leq t_1$ such that $\Phi_{t_2}(z_0)\in \partial \mathcal A$ and $\Phi_{u}(z_0)\in \mathcal A$ for all $u<t_2$. If $t_2=t_1$ we conclude the proof by using the claim. If $t_2<t_1$, then $|\arg(\Phi_{t_2}(z_0))-\theta_0|=\pi/2$ and $|\Phi_{t_2}(z_0)-z_0|\geq \sqrt 2/2$. Moreover, $|\Phi_{t_2}(z_0)-z_0|\leq C\sqrt{t_2}$. Thus, for $t_2<t\leq t_1$ we have
$$
|\Phi_{t}(z_0)-z_0|\leq 2\leq 2\sqrt 2 C\sqrt{t_2}< 2\sqrt 2 C\sqrt{t}.
$$
And the proof would be finished with the constant $2\sqrt 2 C$.

 Thus, let us prove the claim.
Let $T(z)=-\log(z)$, where $\log$ denotes the branch in $\mathcal A$ of the logarithm such that $\log(z_0)=\ln|z_0|+i\theta_0$, and set
$w_0=-\log(z_0)$. Notice that 
$$\mathcal{R}:=T(\mathcal{A})=\{w\in\C_+:0<\text{Re}(w)<\log2, \    -\pi/2-\theta_0<\Im (w)<\pi/2-\theta_0 \}.$$

Define $p_1:\C_{+}\to\overline{\C_+}$ by $p_1(w)=p(e^{-w})$. Hence, it is the infinitesimal generator of a continuous semigroup in $\C_{+}$ (see \cite[Theorem 2.6]{porta}). Let us denote such a semigroup by $\{\Psi_{t}\}$. Furthermore, the generator $p_1$ satisfies the hypothesis of Theorem \ref{convunifsem}. Indeed, by Lemma \ref{lemapre3}, there is $M>0$ such that 
\[
|p_1(w)|=|p(e^{-w})|\leq\frac{M}{1-e^{-\text{Re}(w)}}\leq\frac{2M}{\text{Re}(w)}, 
\]
whenever $0<\Re w<1$. Moreover, there exists $0<\varepsilon_0<1$ such that $|p(z)|<2M/\varepsilon_0$ if $|z|<e^{-1}$. Hence if $\varepsilon<\varepsilon_0$ and $w\in \C_\varepsilon$, we have $|p_1(w)|\leq 2M/\varepsilon$. 

Consider the curve  $\gamma:[0,t)\to\mathcal{R}$ given by $\gamma(t)=-\text{log} (\Phi_t(z_0))$. Then a simple computation yields
\[
\gamma'(t)=-\frac{G(\Phi_t(z_0))}{\Phi_t(z_0)}=p(\Phi_t(z_0))=p(\exp(-\gamma(t)))=p_{1}(\gamma(t)).
\]
This means that $\Psi_{t}(w_0)=\gamma(t)$. 
Therefore, applying Theorem \ref{convunifsem} and using the fact that the function $e^{-w}$ is Lipschitz in $\C_+$, there is a constant $A=A(M)$ and $t_{0}$ such that \[
A\sqrt{t}\geq |\Psi_{t}(w_0)- w_0|=| \text{log} (\Phi_t(z_0))- \text{log} (z_0)|\geq|\Phi_t(z_0)-z_0|,
\]
whenever $t<t_{0}$, and the claim follows.
\end{proof}

\begin{proof}[Proof of the non-elliptic case in Theorem \ref{conunifelip}] Take a non-elliptic continuous semigroup $\{\Phi_t\}$.  We may assume that its Denjoy-Wolff point is $1$ and again it is enough to get the thesis when $t$ is close to zero. Let us call $H$ the infinitesimal generator of the semigroup. By the Berkson-Porta decomposition (see Theorem \ref{BPformula}), there is $p:\D\to \overline{\C_{+}}$ holomorphic such that $H(z)=(z-1)^{2}p(z)$. Moreover, by  Lemma \ref{lemapre3}, there is $M>0$ such that, for all $0<\lambda<1$
    \begin{equation*}
        |H(z)|\leq\frac{M}{\lambda}, \quad  z\in D(\lambda,1-\lambda).
         \end{equation*}

Take $\lambda=\sqrt{t}\leq1/2$ and $z\in \D$. We recall that, by Julia's Lemma \cite[Theorem 1.4.7]{manoloal}, $H_{\lambda}=D(\lambda,1-\lambda)$ is invariant by the elements of the semigroup.  Write $t_0=t_0(z)=\inf\{u:\Phi_u(z)\in H_{\lambda} \}$. If $t_{0}<t$, then 
\[
|\Phi_t(z)-\Phi_{t_{0}}(z)|\leq\int_{t_{0}}^t|H(\Phi_{\tau}(z))|d\tau\leq M\frac{t-t_{0}}{\sqrt{t}}.
\]
Thus, if we prove that there is $C$ such that $|\Phi_u(z)-z|\leq C\sqrt{u}$ for all $u<t_{0}$, then
$$
|\Phi_t(z)-z|\leq \max{\{M,C\}}(\sqrt{t_{0}}+\frac{t-t_{0}}{\sqrt{t}})\leq 2\max{\{M,C\}}\sqrt{t}
$$ and the proof would be concluded. This means that we may assume that $t<t_{0}$, that is, we assume that the trajectory $y(u)=\Phi_u(z)$, $0\leq u\leq t$ lies outside the horodisc $H_{\lambda}$. 

Let $T:\D\to\C_+$, $T(s)=\frac{1+s}{1-s}$, $s\in \D$. We set $A=T(z)$ and $B=T(\Phi_t(z))$. Notice that $T(H_{\lambda})=\C_{\frac{\lambda}{1-\lambda}}$. Moreover, since we took $\lambda\leq1/2$, we also have that $\C_{\frac{\lambda}{1-\lambda}}\supset\C_{2\lambda}$.  Applying again Julia's Lemma, the map $u\mapsto \Re (\Phi_{u}(z))$ is non-decreasing. We know that $\text{Re}(B)-\text{Re}(A)\leq2\lambda$ and call $d=|\text{Im}(A)-\text{Im}(B)|$. We also may assume that $\Im (A)\leq \Im (\Phi_{u}(z))\leq \Im (A)+d=\Im (B)$, for all $0\leq u\leq t$.

Let $G_1(w)=2p(T^{-1}(w))$,  $w\in \C_{+}$.
Set $\gamma:[0,t)\to\C_+$, where $\gamma(u)=T(y(u))$.  $G_{1}$ is the infinitesimal generator of a semigroup in the right half-plane $\C_{+}$ and
$$
\gamma'(u)=T'(y(u))y'(u)=2p(y(u))=G_{1}(\gamma(u)),
$$
for all $u$. Thus, $\gamma$ is a trajectory of such infinitesimal generator.
Notice that  $\gamma^*=\gamma([0,t])$ lies in the vertical strip $\Omega_{2}=\{s\in\C_+:\text{Re}(s)<2\lambda \}$. Now, a simple computation shows that
\[
|\Phi_t(z)-z|=|T^{-1}(A)-T^{-1}(B)|=2\frac{|A-B|}{|1+A||1+B|}.
\]
Denote by $\rho$ the constant provided by Corollary \ref{Lavrentievnoi} and $M$ the constant associated with $p$ given in Lemma \ref{lemapre3}. Take $\eta\geq \max\{18, 24/\rho, 25 M 384/\rho\}$. If $ |\Phi_t(z)-z|<\eta\sqrt t$, then we are done. Assuming on the contrary that    $ |\Phi_t(z)-z|> \eta\sqrt t$,  we will get a contradiction. Notice that in such a case
\begin{equation}\label{negteor}
    |A-B|>\frac12\eta\sqrt{t}|1+A||1+B|\geq \frac{\sqrt{t}}{2}\eta =\lambda \eta/2.
\end{equation}
 Then, recalling that $\eta\geq 18$ and $\Re (A)-\Re (B)\leq 2\lambda$, 
 $$
 d^{2}> \frac{\lambda^{2}}{4}\eta^{2}-(\text{Re}(A)-\text{Re}(B))^{2}\geq  \frac{\lambda^{2}}{4}\eta^{2}-4\lambda ^{2}=\lambda ^{2}(\eta^{2}/4-4)
\geq 64 \lambda ^{2}.
 $$
 That is, 
 \begin{equation}\label{negteorbis}
 d>8\lambda \geq 4(\text{Re}(B)-\text{Re}(A)). 
 \end{equation}
 Once again, we argue as in the proof of Theorem  \ref{convunifsem}. Consider the square
 \[
\mathcal{C}=\{w\in\C_+:\text{Re}(A)<\text{Re}(w)<\text{Re}(A)+d, \  \text{Im}(A)< \text{Im}(w)< \text{Im}(A)+d\}.
\]
Let $w_0=A+(1+i)d/2$. Similarly, we  define the curve $\alpha:[0,t]\to\C_+$ given by $\alpha(u)=f(u)+ig(u)$, where $f(u)=\text{Re}(\gamma(u))$ and $g(u)=\min\{\text{Im}(\gamma(x)): x\geq u\}$. 
We shall also denote by $\gamma^*$ the set $\gamma([0,t])$ and by $\alpha^*$ the set $\alpha([0,t])$. The curve $\alpha^*$ divides the square $\mathcal{C}$ into two regions. Let $\Omega$ be the right hand-side region of the square.

 Then, we introduce the sets
\[
\mathcal{A}_0=\{s\in\alpha^*:s\not\in\gamma^*\},\quad \mathcal{A}_1=\{s\in\alpha^*\cap\gamma^*:|G_1(s)|\leq\frac{\rho}{3} \frac{|A-B|}{t}\}
\]
and  $\mathcal{A}=\mathcal{A}_0\cup\mathcal{A}_1$. For $E\subset\partial\Omega$, $\omega(E)=\omega_{\Omega}(w_0,E)$. Arguing as in the proof of Theorem \ref{convunifsem}, we get that $\omega(\alpha^*)\geq1/4$. We now estimate $\ell(\mathcal{A})$ in order to apply Corollary \ref{Lavrentievnoi}:
\begin{align*}
    \ell(\mathcal{A})&\leq \int_0^{t}\Re \, \alpha'(u) du + \int_{\alpha^{-1}(\mathcal A_1)}\Im\,  \alpha'(u) du\\
    &\leq2\lambda+\frac{\rho}{3} |A-B| \leq \left(\frac{4}{\eta}+\frac{\rho}{3}\right) |A-B|\leq 2 \left(\frac{4}{\eta}+\frac{\rho}{3}\right)d,
\end{align*}
where we have used \eqref{negteor} and \eqref{negteorbis}. Since $\eta>24/\rho$, then $\ell(\mathcal{A})\leq \rho d$.
Of course, $D(w_0,d/4)\subset \Omega$. By Corollary \ref{Lavrentievnoi}, we have that $\omega(\mathcal{A})<1/8$. Since $\omega(\alpha^*)\geq1/4$, we conclude that
\begin{equation}\label{18omega4}
    \omega(\alpha^*\setminus A)>\frac18.
\end{equation}
Using \eqref{18omega4}, both the definitions of the set $\mathcal{A}$ and of harmonic measure, and using that  $\text{Re}(\sqrt{G_1})$ is a positive harmonic function, we get 
\begin{equation*}
\begin{split}
 \sqrt{|G_1(w_0)|}&\geq \Re \sqrt{G_{1}(w_0)}=\int_{\partial\Omega}\text{Re}(\sqrt{G_{1}(s)})d\omega(s)\geq \int_{\alpha^*\setminus \mathcal{A}}\text{Re}(\sqrt{G_{1}(s)})d\omega(s) \\
 &\geq\int_{\alpha^*\setminus \mathcal{A}} \frac{\sqrt{|G_{1}(s)|}}{\sqrt{2}}d\omega(s) \geq  \frac{1}{\sqrt{2}}\sqrt{\frac{\rho}{3} \frac{|A-B|}{t} }\omega(\alpha^*\setminus \mathcal{A})\geq  \frac{1}{\sqrt{2}}\sqrt{\frac{\rho}{3} \frac{|A-B|}{t}} \frac18.
\end{split}
\end{equation*}
That is,
\begin{equation}\label{loweslav4}
|G_1(w_0)|\geq \frac{\rho}{384}\frac{|A-B|}{t}.
\end{equation}

Let us first assume that $1/2\leq |1+A||1+B|^{-1}\leq2$. Using both   \eqref{loweslav4} and \eqref{negteor}, we have that
\[
|G_1(w_{0})|>  \frac{\rho}{384} \frac12\eta \frac{1}{\sqrt{t}}|1+A||1+B|.
\]
We recall that $w_0=x+iy$ is the centre of the square $\mathcal{C}$. Now, $|1+A|\, |1+B|\geq N/2$, where 
 $N=\max(|1+A|^2,|1+B|^2)$. Then, since $\text{Im}(A)<y<\text{Im}(B)$, we have $N\geq y^2$. The definition of $x$ together with the fact that $d>8\lambda$ imply that
\begin{equation}\label{xd4}
   x<2\lambda+\frac{d}{2}\leq\frac34d<\frac32\max(|\text{Im}(A)|,|\text{Im}(B)|)\leq\frac32\sqrt N. 
\end{equation}
In particular, $1+x\leq \frac{5}{2}\sqrt N$. Using these estimates on $x$ and $y$, we find that $8N>y^2+(1+x)^2$. Therefore, 
$$
|G_1(w_0)|>  \frac{\rho}{384} \frac{1}{32}\eta \frac{1}{\sqrt{t}}(y^2+(1+x)^2)\geq \frac{\rho}{384} \frac{1}{16}\eta \frac{y^2+(1+x)^2}{x}
$$
where we have used $2\sqrt{t}=2\lambda<x$. Finally, Lemma \ref{lemapre3} c) gives
$$
|G_1(w_0)|>\frac{\rho}{384} \frac{1}{8M}\eta |G_1(w_{0})|.
$$
Since $\eta> 8M \frac{384}{\rho}$, we obtain a contradiction and, consequently, it cannot occur that $1/2\leq |1+A||1+B|^{-1}\leq2$.

Suppose now that $|1+B|>2|1+A|$. The case $2|1+B|<|1+A|$ is proven in a same fashion. Using Lemma \ref{lemapre3} we have that
\begin{equation}\label{consqlema}
 |G_1(w_0)|\leq M\frac{(1+x)^2+y^2}{2x}.   
\end{equation}
Notice that, by our assumption,
\[
|A-B|\geq|1+B|-|1+A|\geq\frac{|1+B|}{2}.
\]
Therefore, we have $x\geq d/2\geq|1+B|/8$ and
$$
|y|\leq \frac{|\Im \, A|+|\Im \, B|}{2}\leq \frac{|1+ A|+|1+B|}{2}<|1+B|. 
$$
Now, as $|1+B|>2|1+A|$,
\[
d=|\text{Im}(A)-\text{Im}(B)|\leq |\text{Im}(A)|+|\text{Im}(B)|
\leq|1+A|+|1+B|\leq\frac32|1+B|.
\]
Eventually, using again the fact that $x\leq\frac34d$ (see \eqref{xd4}), we have $x\leq 9|1+B|/8$. Using these last three estimates in \eqref{consqlema}, we find that
\[
|G_1(w_0)|\leq M\frac{(1+x)^2+y^2}{2x}<25M |1+B|.
\]
On the other hand, using again \eqref{loweslav4}  and  \eqref{negteor}, we have that
\[
|G_1(w_0)|>  \frac{\rho}{384} \frac12\eta \frac{1}{\sqrt{t}}|1+A||1+B|\geq \frac{\rho}{384} \eta |1+A||1+B|\geq \frac{\rho}{384} \eta |1+B| ,
\]
where we have used that $\sqrt{t}\leq1/2$. Since $\eta> 25 M \frac{384}{\rho}$, we reach the desired contradiction.
\end{proof}

To finish we provide an example of how the statement of Theorem \ref{conunifelip} cannot be strengthened in the sense that $t^{1/2}$ cannot  be replaced  by another function that goes to zero faster than  $t^{1/2}$.
\begin{example} \label{Ex:cotaoptima}
Consider $p:\D\to\C_+$, $p(z)=(1+z)^{-1}$, $z\in \D$ and  the infinitesimal generator given by
\[
H(z)=(1-z)^{2}p(z)=\frac{(1-z)^2}{1+z}, \quad z\in \D.
\]
Let $\{\Phi_t\}$ be its continuous semigroup. Its Denjoy-Wolff point is $1$ and since $H(x)$ is real whenever $x\in (-1,1)$, we have that $\Phi_{t}(x)\in (-1,1)$ for all $x\in (-1,1)$.

Consider the vector field $G(z)=\frac{1}{1+z}$, with $z\in (-1,0)$. For $x\in (-1,0)$ fixed, the solution of the initial value problem 
\[
y'(t)=G(y(t)),\quad  y(0)=x
\]
is given by $y_{x}(t)=\sqrt{2t+(1+x)^{2}}-1$ for all $t>0$.

We claim that for every $t>0$ and $x<0$ such that $\Phi_t(x)<0$, we necessarily have that $\Phi_t(x)\geq y_x(t)$. Fix $t>0$ and $x<0$  so that $\Phi_t(x)<0$. Then, we have that $\Phi_s(x)<0$ for all $0\leq s\leq t$. Define  $e(u)= y_{x}(u)-\Phi_{u}(x)$. Observe that $e(0)=0$ and 
\begin{equation}\label{derivada}
    e'(u)=
\frac{\partial}{\partial t}\left(y_x(t)
    -
   \Phi_t(x)\right)\Big\lvert_{t=u}
    =G(y_x(u))
    -
     H(\Phi_{u}(x)).
\end{equation}
Since $H(z)>G(z)$ for all $z\in  (-1,0)$, we have that $e'(0)<0$. This, together with the fact that $e(0)=0$, imply that $e(u)<0$ for $u$ small enough. This is, $\Phi_u(x)\geq y_x(u)$ for $u$ small enough.

\ 
Let us now suppose the existence of $u<t$ such that $y_x(u)>\Phi_u(x)$. Define
\[
t_1=\inf\{u>0: \Phi_u(x)<0, \ e(u)=0\}. 
\]
The assumption on $u$ and the fact that $e(w)<0$ for $w$ small enough imply that $t_1<t$. Thus, for $u<t_1$, $e(u)<0$. Then, $e'(t_1)\geq0$. But,
\[
e'(t_1)=G(y_x(t_1))-H(\Phi_{t_1}(x))<0
\]
since $H(u)>G(u)$ for all $u\in(-1,0)$. A contradiction. Therefore, $\Phi_{t}(x)\geq y_{x}(t)$ for  all $x\in (-1,0)$ and all $t$ such that $\Phi_{t}(x)<0$.

\ 
Notice that the curve $t\mapsto \Phi_t(-1/2)$ is contained in the interval $(-1,1)$ and $\lim_{t\to\infty}\Phi_t(-1/2)=1$. Then, we can take $t_{0}$ such that $\Phi_{t_{0}}(-1/2)=0$. Then $\Phi_{t}(x)<0$ for all $x<-1/2$ and $t<t_{0}$. Thus, for $t<t_{0}$
\begin{equation*}
\begin{split}
\sup_{z\in \D}|\Phi_{t}(z)-z|&\geq \sup_{x<-1/2}|\Phi_{t}(x)-x|=\sup_{x<-1/2}(\Phi_{t}(x)-x)\\
&\geq \sup_{x<-1/2}(y_{x}(t)-x)=\sup_{x<-1/2}(\sqrt{2t+(1+x)^{2}}-1-x).
\end{split}
\end{equation*}
If $x<\sqrt t/2-1$, then $\sqrt{2t+(1+x)^{2}}-1-x\geq\sqrt t$. Therefore, $$\sup_{z\in \D}|\Phi_{t}(z)-z|\geq \sqrt t,\quad \text{for all $t<t_{0}$},$$ and the claim follows.
\end{example}


\begin{thebibliography}{}



\bibitem{Avicou} C. Avicou, I. Chalendar, and J.R. Partington, {\sl Analyticity and compactness of semigroups of composition operators.} J. Math. Anal. Appl. {\bf 437} (2016), 545--560.




%


\bibitem{porta} E. Berkson and H. Porta, {\sl Semigroups of analytic functions and composition operators}, Michigan Math. J. \textbf{25} (1978), 101--115.

\bibitem{Bet15a} D. Betsakos, \textsl{On the rate of convergence of parabolic semigroups of holomorphic functions.} Anal. Math. Phys., \textbf{5} (2015), 207--216.

\bibitem{BCD} D. Betsakos, M.D. Contreras, and S.  D\'{\i}az-Madrigal, {\sl On the rate of convergence of semigroups of holomorphic functions at the Denjoy-Wolff point.} Rev. Mat. Iberoam. {\bf 36} (2020), 1659--1686.

\bibitem{BCDMPS}  O. Blasco, M.D. Contreras, S. D\'{\i}az-Madrigal, J. Mart\'{\i}nez, M. Papadimitrakis, and A.G. Siskakis, {\sl Semigroups of composition operators and integral operators in spaces of analytic functions},  Ann. Acad. Sci. Fenn. Math. \textbf{38} (2013), 67--89.

\bibitem{BCDGZ} F. Bracci, M.D. Contreras, S. D\'{\i}az-Madrigal, H. Gaussier, and A. Zimmer, 
{\sl Asymptotic behavior of orbits of holomorphic semigroups}. J. Math. Pures Appl. (9) {\bf133} (2020), 263--286.



\bibitem{manoloal}F. Bracci, M.D. Contreras, and S. D\'{\i}az-Madrigal, \emph{Continuous semigroups of holomorphic self-maps of the unit disc}, Springer, Berlin, 2020.

\bibitem{BCK} F. Bracci, D. Cordella, and M. Kourou, {\sl Asymptotic monotonicity of the orthogonal speed and rate of convergence for semigroups of holomorphic self-maps of the unit disc}. Rev. Mat. Iberoam. {\bf 38} (2022),  527--546.

\bibitem{ChaPar} I. Chalendar and J.R. Partington, {\sl Semigroups of weighted composition operators on spaces of holomorphic functions.} Available on https://arxiv.org/pdf/2201.09249.pdf.

\bibitem{Contreras-Diaz} M. D. Contreras and S. D\'{\i}az-Madrigal, {\sl
Fractional iteration in the disk algebra: prime ends and composition
operators,} Rev. Mat. Iberoam. \textbf{21} (2005), 911--928.

\bibitem{noi} M.D. Contreras, C. Gómez-Cabello, and L. Rodríguez-Piazza, {\sl Semigroups of composition operators in Hardy spaces of Dirichlet series}. Journal of Functional Analysis, {\bf 9} (2023), article number 110089. 


\bibitem{noi2} M.D. Contreras, C. Gómez-Cabello, and L. Rodríguez-Piazza, {\sl On the uniform convergence of continuous semigroups}. Available on {ArXiv:2304.12759}


  
  
  

\bibitem{peris}A. Defant, D. Garc\'{\i}a, M. Maestre, and P. Sevilla-Peris, \emph{Dirichlet Series and Holomorphic
Functions in High Dimensions}, Cambridge University Press, Cambridge, 2019.




\bibitem{gorheda}J. Gordon and H. Hedenmalm,  {\sl The composition operators on the space of Dirichlet series with square summable coefficients}, Michigan Math. J. \textbf{49} (1999), 313--329.

\bibitem{pavel}P. Gumenyuk, {\sl Angular and unrestricted limits of one-parameter semigroups in the unit disk}, J. Math. Anal. Appl. \textbf{417}  (2014), 200–224.


\bibitem{HedLinSeip}H. Hedenmalm, P. Lindqvist, and K. Seip, {\sl A Hilbert space of Dirichlet series and systems of dilated functions in $L^{2}(0,1)$,} Duke Math. J. {\bf 86} (1997), 1--37.

\bibitem{Lav} M. Lavrentiev, {\sl Boundary problems in the theory of univalent functions} (Russian), Mat. Sb. (N.S.) {\bf 1} (1936), 815--844; Amer. Math. Soc. Transl. (2) {\bf 32} (1963), 1--35.

\bibitem{Pommerenke} Ch. Pommerenke, \emph{Boundary Behaviour of Conformal Mappings}, Springer, Berlin, New York, 1992.

\bibitem{queffelecs} H. Queffélec and M. Queffélec, \emph{Diophantine approximation and Dirichlet series}, Hindustan Book Agency and Springer, Singapore, 2020.

\bibitem{Ran} T. Ransford, {\sl Potential Theory in the Complex Plane},
Cambridge University Press, 1995.





\end{thebibliography}
\end{document}